\documentclass[svgnames]{amsart}
\usepackage{
	amssymb,
	mathtools,
	xcolor,
	enumitem,
	booktabs,
	hyperref,
	cleveref,
}

\setlist[enumerate]{label=\textup{(\roman*)}}

\hypersetup{colorlinks, allcolors=MediumBlue}

\crefdefaultlabelformat{#2\textup{#1}#3} 

\newtheorem*{thm*}{Theorem}
\newtheorem*{ques*}{Question}
\newtheorem{thm}  {Theorem} [section]
\newtheorem{prop} [thm] {Proposition}
\newtheorem{lem} [thm] {Lemma}
\newtheorem{cor} [thm] {Corollary}
\theoremstyle{definition} 
\newtheorem*{rmk*}{Remark}

\numberwithin{equation}{section}

\crefname{thm}{Theorem}{Theorems}
\crefname{lem}{Lemma}{Lemmas}
\crefname{prop}{Proposition}{Propositions}
\crefname{cor}{Corollary}{Corollaries}

\DeclarePairedDelimiterX\set[1]\lbrace\rbrace{\,\def\given{\mid}#1\,} 

\newcommand{\cv}[1]{\ensuremath{\operatorname{cv}(#1)}}
\newcommand{\cd}[1]{\ensuremath{\operatorname{cd}(#1)}}
\newcommand{\Irr}[1]{\ensuremath{\operatorname{Irr}(#1)}}
\newcommand{\Lin}[1]{\ensuremath{\widehat{#1}}}
\DeclareMathOperator{\Dih}{Dih}
\newcommand{\size}[1]{\ensuremath{\lvert #1 \rvert}}
\renewcommand{\Re}{\operatorname{Re}}

\begin{document}

\title{Finite groups with very few character values}
\author[T.~Sakurai]{\href{https://orcid.org/0000-0003-0608-1852}{Taro Sakurai}}
\address{Department of Mathematics and Informatics, Graduate School of Science, Chiba University, 1-33, Yayoi-cho, Inage-ku, Chiba-shi, Chiba, 263-8522 Japan}
\email{tsakurai@math.s.chiba-u.ac.jp}

\subjclass[2010]{%
	\href{https://zbmath.org/classification/?q=20C15}{20C15} (%
	\href{https://zbmath.org/classification/?q=20D60}{20D60})}
\date{\today}
\begin{abstract}
	\noindent 
	Finite groups with very few character values are characterized.
	The following is the main result of this article:
	a finite non-abelian group has precisely four character values if and only if it is the generalized dihedral group of a non-trivial elementary abelian $3$-group.
	The proof involves the analysis of the centralizers of involutions.
\end{abstract}

\maketitle


\section*{Introduction}
\noindent 
A character table often contains substantial information about a finite group.
Even its partial data affects considerably on the group structure,
and character degrees has been drawn special attention over the years.
Surprisingly, however, the author cannot find any study on how the set of \emph{character values}
\begin{equation*}
	\cv{G} = \set{ \chi(g) \given \chi \in \Irr{G},\ g \in G }
\end{equation*}
affects the structure of a finite group \( G \),
where \( \Irr{G} \) denotes the set of irreducible characters of \( G \).

\begin{ques*}
	What information about a finite group can be obtained from its character values?
\end{ques*}

This article presents the characterizations of finite groups with very few character values.
Finite groups with less than four character values are abelian and their characterizations are rather easy (\cref{prop:very few}).
The first non-trivial example is, as in many cases, the symmetric group \( S_3 \) of degree three.
\begin{table}[h]
	\centering
	\newcommand{\columnlength}{9ex}
	\begin{tabular}{p{\columnlength}p{\columnlength}p{\columnlength}p{\columnlength}}
		\toprule
					        & \( (1)(2)(3) \)    & \( (1, 2)(3) \)     & \( (1, 2, 3) \)    \\
		\midrule
		\( \chi^{(3)} \)    & \( \phantom{-}1 \) & \( \phantom{-}1 \)  & \( \phantom{-}1 \) \\
		\( \chi^{(1^3)} \)  & \( \phantom{-}1 \) &           \( -1 \)  & \( \phantom{-}1 \) \\
		\( \chi^{(1, 2)} \) & \( \phantom{-}2 \) & \( \phantom{-}0 \)  &           \( -1 \) \\
		\bottomrule
	\end{tabular}
\end{table}
(See the character table.)
It is non-abelian and has precisely four character values.
This is the case we settle in this article.

\begin{thm*}
	A finite non-abelian group has precisely four character values if and only if it is the generalized dihedral group of a non-trivial elementary abelian \( 3 \)-group.
\end{thm*}

The \emph{generalized dihedral group} \( \Dih A \) of \( A \) is defined to be the relative holomorph of \( A \) by the inversion \( A \to A \), \( a \mapsto a^{-1} \), where \( A \) is an abelian group that is not an elementary abelian \( 2 \)-group.
In particular, the generalized dihedral group \( \Dih C_3^r \) of a non-trivial elementary abelian \( 3 \)-group \( C_3^r \) has a presentation
\begin{equation*}
	\langle\, a_1, \dotsc, a_r, t \mid a_i^3 = [a_i, a_j] = t^2 = 1,\ a_i^t = a_i^{-1} \,\rangle.
\end{equation*}
Its centralizers of involutions have order two;
this is the key to identify the isomorphism classes of finite non-abelian groups with precisely four character values.

\begin{rmk*}
	It is natural to wonder what happens if a finite group has more character values.
	One can easily check that
	\( C_5^* \), \( C_2^* \times \Dih C_3^* \), and \( S_4 \)
	have precisely five character values where \( \ast \) indicates some positive integer.
	But there are also \emph{many} finite \( 2 \)-groups with precisely five character values,
	including the dihedral group \( D_8 \) and the quaternion group \( Q_8 \) of order eight.
	This makes it hard to prove or even guess the next case.
	Having said that, the character values seem to impose severe restriction on the group structure.
	For instance, it appears that finite groups with less than eight character values are solvable.
\end{rmk*}

We refer the readers to the books by Isaacs~\cite{Isa94} and Huppert~\cite{Hup98} for standard terminology and notation in the character theory of finite groups.

\section{Character values}
The preliminary results on character values are summarized in this section.
The largest character degree of a finite group \( G \) is denoted by \( b(G) \).

\begin{lem}
	\label{lem:zero}
	A finite group is non-abelian if and only if it has zero as a character value.
	In particular, every finite group with less than four character values is abelian.
\end{lem}

\begin{proof}
	See \cite[Theorem~3.15 (Burnside)]{Isa94} for the first part.
	For the second part, we shall prove the contraposition.
	Let \( G \) be a finite non-abelian group.
	We have \( \cv{G} \supseteq \{ 0, 1, b(G) \} \) by the first part.
	From the column orthogonality, there must be a character value with negative real part.
	Hence \( \size{\cv{G}} \ge 4 \).
\end{proof}

\begin{lem}
	\label{lem:|cv(G)|}
	The number of character values of a finite abelian group equals the maximum order of elements.
\end{lem}

\begin{prop}
	\label{prop:very few}
	Let \( G \) be a finite group.
	\begin{enumerate}
		\item \( \size{\cv{G}} = 1 \) if and only if \( G \) is trivial. \label{item:one}
		\item \( \size{\cv{G}} = 2 \) if and only if \( G \) is a non-trivial elementary abelian \( 2 \)-group. \label{item:two}
		\item \( \size{\cv{G}} = 3 \) if and only if \( G \) is a non-trivial elementary abelian \( 3 \)-group. \label{item:three}
		\item Assume \( G \) is abelian. Then \( \size{\cv{G}} = 4 \) if and only if \( G \) has exponent \( 4 \). \label{item:four}
	\end{enumerate}
\end{prop}

\begin{proof}
	The assertions follow from \cref{lem:zero,lem:|cv(G)|}.
\end{proof}

\section{Generalized dihedral groups}
The easy part of Theorem is proved in this section.
For a finite group \( G \),
the principal character of \( G \) is denoted by \( 1_G \),
the set of linear characters of \( G \) is denoted by \( \Lin{G} \), and
the largest normal subgroup of \( G \) of odd order is denoted by \( O(G) \).

\newpage 
\begin{prop}
	\label{prop:character table}
	Let \( A \) be a non-trivial finite abelian group of odd order,
	\( G \) the generalized dihedral group of \( A \), and
	\( t \) an involution of \( G \).
	Then the character table of \( G \) is given by the following.
	\begin{table}[h]
		\centering
		\newcommand{\columnlength}{9ex}
		\begin{tabular}{p{\columnlength}p{\columnlength}p{\columnlength}p{\columnlength}}
			\toprule
						             & \( 1 \) & \( \phantom{-}t \)  & \( a\phantom{\Re\theta(a)} \) \\
			\midrule
			\( 1_G\)                 & \( 1 \) & \( \phantom{-}1 \)  & \( 1\phantom{\Re\theta(a)} \) \\
			\( \sigma\phantom{_G} \) & \( 1 \) & \( -1 \)            & \( 1\phantom{\Re\theta(a)} \) \\
			\( \theta^G \)           & \( 2 \) & \( \phantom{-}0 \)  & \( 2\Re\theta(a) \) \\
			\bottomrule
		\end{tabular}
	\end{table}

	\noindent
	Here, \( a \) indicate representatives of \( A \setminus \{ 1 \} \) under the equivalence relation generated by \( a \sim a^t \) and
	\( \theta \) indicate representatives of \( \Lin{A} \setminus \{ 1_A \} \) under the equivalence relation generated by \( \theta \sim \theta^t \).
	In particular, \( G \) has \( (\size{A} + 3)/2 \) conjugacy classes, \( G' = O(G) \), and \( C_G(t) = \langle t \rangle \).
\end{prop}

\begin{proof}
	Because \( A \) is non-trivial group of odd order, \( G \) is a Frobenius group with the Frobenius kernel \( A \) and a Frobenius complement \( \langle t \rangle \).
	Then the assertions follow from \cite[Theorem~18.7]{Hup98}.
\end{proof}

\begin{cor}
	\label{cor:Dih C_3^r}
	\( \Dih C_3^r \) has precisely four character values for every \( r \ge 1 \).
\end{cor}

\section{Proof of Theorem}
The difficult part of Theorem is proved in this section.
For a finite group \( G \),
the set of character degrees of \( G \) is denoted by \( \cd{G} \).

\begin{lem}
	\label{lem:values}
	Let \( G \) be a finite non-abelian group with precisely four character values.
	\begin{enumerate}
		\item \( \cd{G} = \{ 1, b(G) \} \). \label{item:cd(G)}
		\item \( b(G) \) divides \( \size{G : G'} \). \label{item:b(G)}
		\item \( \cv{G} = \{ -1, 0, 1, b(G) \} \). \label{item:cv(G)}
	\end{enumerate}
\end{lem}

\begin{proof}
	\ref{item:cd(G)}
	As \( G \) is non-abelian, \( b(G) \neq 1 \).
	By \cref{lem:zero}, we have \( \cv{G} \supseteq \{ 0, 1, b(G) \} \).
	Suppose that \( \psi(t) \in \cv{G} \) is the remaining character value.
	From the row orthogonality
	\(
		0 = \langle \psi, 1_G \rangle = \frac{1}{\size{G}}\sum_{g \in G} \psi(g),
	\)
	the remaining character value \( \psi(t) \)  must be a negative rational number.
	Hence \( \cd{G} = \{ 1, b(G) \} \).

	\ref{item:b(G)}
	As \( b(G) \) divides
	\(
		\size{G} = \sum_{\chi \in \Irr{G}} \chi(1)^2 = \size{G : G'} + b(G)^2 + \dotsb + b(G)^2,
	\)
	it follows that \( b(G) \) divides \( \size{G : G'} \).

	\ref{item:cv(G)}
	As \( \size{G : G'} \neq 1 \) by the part~\ref{item:b(G)}, there is a non-principal linear character \( \sigma \).
	Since values of \( \sigma \) are complex numbers of modulus one, we get \( \psi(t) = -1 \).
	Hence \( \cv{G} = \{ -1, 0, 1, b(G) \} \).
\end{proof}

\begin{lem}
	\label{lem:local}
	Let \( G \) be a finite non-abelian group with precisely four character values.
	\begin{enumerate}
		\item \( G/G' \) is a non-trivial elementary abelian \( 2 \)-group. \label{item:elementary}
		\item \( C_G(t) \) is a \( 2 \)-group for every involution \( t \) of \( G \). \label{item:centralizer}
		\item \( G/O(G) \) is a \( 2 \)-group. \label{item:2-nilpotent}
	\end{enumerate}
\end{lem}

\begin{proof}
	\ref{item:elementary}
	It follows from \cref{lem:values} and \cref{prop:very few}.

	\ref{item:centralizer}
	Let \( t \) be an involution of \( G \).
	We claim that
	\begin{equation}
		\chi(t) = 0 \tag{\textdagger}\label{eq:claim}
	\end{equation}
	for every non-linear character \( \chi \) of \( G \).
	Suppose that \( \mathcal X \) is a representation of \( G \) affording \( \chi \).
	Because \( \mathcal X(t) \) is involutory, its eigenvalues are \( \pm 1 \).
	As \( b(G) \) is even,
	\( \chi(t) \) equals \( 0 \) or \( b(G) \) by \cref{lem:values}.
	From the column orthogonality,
	\begin{equation*}
		0
		= \sum_{\psi \in \Irr{G}} \psi(t)\psi(1)
		= \sum_{\sigma \in \Lin{G}} \sigma(t) + b(G)\sum_{\mathclap{\psi \in \Irr{G} \setminus \Lin{G}}} \psi(t)
		\ge \sum_{\sigma \in \Lin{G}} \sigma(t).
	\end{equation*}
	(Note that \( \psi(t) \) equals \( 0 \) or \( b(G) \) for \( \psi \in \Irr{G} \setminus \Lin{G} \).)
	Hence \( \sigma(t) = -1 \) for some non-principal linear character \( \sigma \). 
	Thus if \( \chi(t) \neq 0 \),
	then \( -b(G) = \sigma \otimes \chi(t) \in \cv{G} \), a contradiction.

	Therefore, by the claim~(\ref{eq:claim}) and the column orthogonality,
	\begin{equation*}
		\size{C_G(t)} = \sum_{\chi \in \Irr{G}} \size{\chi(t)}^2 = \sum_{\sigma \in \Lin{G}} \size{\sigma(t)}^2 = \size{G : G'}.
	\end{equation*}
	It follows from the part~\ref{item:elementary} that \( C_G(t) \) is a \( 2 \)-group.

	\ref{item:2-nilpotent}
	Since every degree of a non-linear character is even,
	it follows from \cite[Corollary 12.2 (Thompson)]{Isa94} that \( G \) is \( 2 \)-nilpotent.
\end{proof}

\begin{lem}
	\label{lem:global}
	Let \( G \) be a finite non-abelian group with precisely four character values.
	\begin{enumerate}
		\item \( G' = O(G) \). \label{item:residual}
		\item Every involution of \( G \) acts on \( O(G) \) as the inversion. \label{item:inversion}
	\end{enumerate}
\end{lem}

\begin{proof}
	\ref{item:residual}
	As \( G/G' \) is a \( 2 \)-group, we have \( O(G) \le G' \) by \cref{lem:local}.
	Suppose \( O(G) \neq G' \).
	Then \( \bar G = G/O(G) \) is a non-abelian \( 2 \)-group and \( \size{\cv{\bar G}} = 4 \).
	Hence, by \cite[Corollaries 2.23 and 2.28]{Isa94},
	\(
		{\bar G}' \cap Z(\bar G) = \bigcap_{\chi \in \Irr{\bar G}} \ker \chi = 1
	\).
	This is a contradiction because \( \bar G \) is nilpotent and \( {\bar G}' \) is a non-trivial normal subgroup.

	\ref{item:inversion}
	Let \( t \) be an involution of \( G \).
	As \( C_G(t) \) is a \( 2 \)-group by \cref{lem:local}\ref{item:centralizer},
	it acts fixed-point-freely on \( O(G) \).
	From \cite[Proposition~16.9e]{Hup98}, it follows that every element of \( O(G) \) is inverted by \( t \).
\end{proof}

\begin{proof}[Proof of Theorem]
	As `if' part is already done in \cref{cor:Dih C_3^r}, we shall prove `only if' part.
	Let \( G \) be a finite non-abelian group with \( \size{\cv{G}} = 4 \) and a Sylow \( 2 \)-subgroup \( P \).
	By \cref{lem:local}\ref{item:elementary} and \cref{lem:global}\ref{item:residual},
	it follows that \( P \) is elementary abelian.
	Let \( s \) and \( t \) be non-trivial elements of \( P \).
	By \cref{lem:global}\ref{item:inversion},
	every element of \( O(G) \) is inverted by \( s \) and \( t \).
	Then \( O(G) \) is abelian and \( O(G) \le C_G(st) \).
	It follows that \( C_G(st) \) cannot be a \( 2 \)-group by \cref{lem:global}\ref{item:residual} and thus \( st = 1 \) by \cref{lem:local}\ref{item:centralizer}.	
	Hence, \( P \) has order two and \( G \) is the generalized dihedral group of \( O(G) \).
	By \cref{prop:character table}, \( O(G) \) must be a non-trivial elementary abelian \( 3 \)-group.
\end{proof}

\section*{Acknowledgements}
The author would like to thank Shigeo Koshitani and Chris Parker for helpful comments.


\begin{thebibliography}{9}

\bibitem{Hup98}
	\textsc{B.~Huppert},
	\textsl{Character theory of finite groups},
	(de Gruyter, Berlin, 1998),
	MR  \href{http://www.ams.org/mathscinet-getitem?mr=1645304}{1645304},
	Zbl \href{https://zbmath.org/?q=an:0932.20007}{0932.20007}.

\bibitem{Isa94}
	\textsc{I.~M.~Isaacs},
	\textsl{Character theory of finite groups},
	(Dover, New York, 1994),
	MR  \href{http://www.ams.org/mathscinet-getitem?mr=1280461}{1280461},
	Zbl \href{https://zbmath.org/?q=an:0849.20004}{0849.20004}.

\end{thebibliography}
\end{document}